\documentclass[12pt]{amsart}

\usepackage{parskip}
\usepackage{amsfonts, amssymb, amsmath, amsthm,  mathrsfs, stmaryrd} 
\usepackage[margin=0.7in]{geometry}
\usepackage[all]{xy}
\usepackage{verbatim}

\newcommand{\Z}{\mathbb{Z}}

\DeclareMathOperator{\End}{End}

\newtheorem{theorem}{Theorem}
\newtheorem{Theorem}[theorem]{Theorem}

\newtheorem{definition}[theorem]{Definition}
\newtheorem{corollary}[theorem]{Corollary}
\newtheorem{lemma}[theorem]{Lemma} \newtheorem{prop}[theorem]{Proposition}
\newtheorem{example}[theorem]{Example}

  \numberwithin{theorem}{section}

\input xy
\xyoption{all}

\begin{document}


\author{Jon Cohen}
\title{On rings as unions of four subrings}
\maketitle

\thispagestyle{empty}
\pagestyle{headings}

\begin{abstract}
The covering number of an associative ring $R$ is the minimal number of proper subrings whose union is $R$. We establish a strategy to classify unital rings of a given finite covering number, and obtain a classification of unital rings whose covering number is four. Along the way we compute the covering number of every finite local ring whose residue field has prime order. 
\end{abstract}


\section{Introduction}

A standard problem in introductory abstract algebra courses is to show that a group is never the union of two proper subgroups. However, there exist groups which are the union of three proper subgroups. Indeed, by a theorem of Scorza \cite{Sco26}, these are precisely the groups that surject onto the Klein four group.

Similarly, no ring is the union of two subrings, but can be the union of three or more subrings. If $R$ is an associative ring, we say that $R$ is  {\it coverable} if it is the union of proper subrings, 
which need not have any multiplicative unit even if $R$ does. We call the collection of these subrings a {\it cover} of $R$. For such an $R$, let $\sigma(R)$ be the minimal possible size of a cover of $R$. We write $\sigma(R) = \infty$ if $R$ is not coverable or has no finite cover. 

In \cite{Maro12} it was shown precisely which associative rings are the union of three subrings, i.e., have $\sigma(R)=3$. The main goal of this paper is to extend some of their results to study the case $\sigma(R)=4$; we obtain a complete classification for rings $R$ with unit. 
One attractice aspect of our approach is that it reduces to a finite computation the problem of classifying rings that are a union of a given number of subrings. We now outline our results.

If $R$ is a ring, then we always have $\sigma(R)\leq \sigma(R/I)$ for an ideal $I$, since a cover of $R/I$ pulls back to a cover of $R$. For $n\geq 3$, define the set $S(n)$ to be those unital rings $R$ with $\sigma(R)=n$ and $\sigma(R/I)>n$ for every proper quotient $R/I$. Our first result is the following theorem. 

\begin{Theorem}
The set $S(n)$ is finite. If $R\in S(n)$ then $R$ is finite and has characteristic $p$ for some prime $p<n$. A ring has covering number $n$ if and only if it has a quotient in $S(n)$ and has no quotient in $S(m)$ for any $m<n$. 
\end{Theorem}

This theorem provides one explanation for the fact that, when $n=3$, the rings appearing in \cite{Maro12} all had order a power of 2. A corresponding statement for groups can be found in Theorem 5 of \cite{Bhar09}. Our next  result was the original motivation for the paper, and answers the question of which (unital) rings are the union of four subrings.

\begin{Theorem}
A unital ring $R$ has $\sigma(R)=4$ if and only if $\sigma(R)\neq 3$ and $R$ has a quotient isomorphic to one the rings in examples \ref{p=2 comm}, \ref{p=2 Noncomm}, \ref{examplecomm}, or \ref{examplenoncomm}, 
\end{Theorem}

We now outline the article's contents. In the first section, we define the sets $S(n)$ and establish some of their general properties, including the first result stated above. Here we also prove a useful lemma on covering numbers of products of rings. In the next two sections, we compute the covering numbers for all rings of order at most $p^3$, and all finite local rings with residue field of prime order. We then begin the attack on $S(4)$ in earnest, proving with relative ease the classification of those rings in $S(4)$ of characteristic 3. In the penultimate and most technical section, we classify the rings in $S(4)$ of characteristic 2. Doing this involves first obtaining the bounds $2^4\leq|R|\leq 2^8 $, and then performing a somewhat elaborate case-by-case elimination to show that no rings in $S(4)$ exist besides those already constructed. In the final section we suggest some further directions of research. 

\section{Notation}

We write $\Z_n$ for the ring of order $n$ generated by $1$, and $C_n$ for the cyclic group of order $n$. If $q$ is a prime power, then $F_q$ denotes the field with $q$ elements. For a set $X$ we write $|X|$ for its cardinality. The letter $p$ will always indicate a prime number. For the entirety of this paper, $R$ indicates a unital ring unless otherwise stated, while $S_i$ and $S$ denote subrings of $R$ which are not required to have a multiplicative unit. 

 \section{The set $S(n)$}

\begin{definition}
For $n\geq 3$, define $S(n)$ to be the collection of (isomorphism classes of) unital rings $R$ such that $\sigma(R)=n<\sigma(R/I)$ for all proper ideals $I$ of $R$. For $R\in S(n)$, with maximal subrings $S_1, \ldots, S_n$ such that $R = \bigcup\limits_{i=1}^n S_i$, call the corresponding $(n+1)$-tuple $(R, S_1, \ldots, S_n)$ a {\it good tuple}.  
\end{definition} 

If $R\in S(n)$ and $R = \bigcup\limits_{i=1}^n S_i$ for proper subrings $S_i$, the intersection $S:=\bigcap\limits_{i=1}^n S_i$ does not contain any nonzero 2-sided ideal $I$ of $R$. Otherwise we could replace $R$ and $S_i$ by $R/I$ and $S_i/I$, respectively. 
Ours is a slightly stronger notion of a good tuple then that of \cite{Maro12}, which only requires that $S$ not contain any proper 2-sided ideals. It also did not require the $S_i$ to be maximal, but that is automatic if $n=3$. 

\begin{example}
The set $S(3)$ was computed completely in \cite{Maro12}. Its two elements (we are requiring rings to have unit) are $F_2\times F_2$ and $F_2[x, y]/(x^2, xy, y^2)$ of order 4 and 8, respectively. These were distinguished by whether their maximal subrings all contained $1_R$, and also by their radicals, which have orders 0 and 4, respectively. Note that both are commutative. 
\end{example}

\begin{example} \label{p=2 comm}
Let $R=F_4\times F_4$. There are exactly four maximal subrings: $F_4\times F_2$,  $F_2\times F_4$, the subring $\{ (t,t)  \}_{t\in F_4}$, and the subring $\{ (0,0), (1,1) , (x, x+1), (x+1, x) \}$, where $x\in F_4 \setminus F_2$. Their union is $R$ and no three form a cover, so $R$ is coverable and $\sigma(R)=4$. 
Since no quotient of $R$ is coverable, $R\in S(4)$. 
\end{example}

\begin{example}\label{p=2 Noncomm}
Let $R= M_2(F_2) $. There are exactly four maximal subrings: the unique copy of $F_4$ inside $R$ 
and the stabilizers of the three lines in $F_2\times F_2$. Their union is $R$ and no three form a cover, so $R$ is coverable and $\sigma(R)=4$. Since $R$ has no proper 2-sided ideals, $R\in S(4)$. 
\end{example}

\begin{example}\label{examplecomm}
 Let $R=F_3[x, y]/(x^2, xy, y^2)$. 
There are exactly four maximal subrings: $S_t:=\{a + b t : a, b\in F_3 \}$ where $t \in \{x, y, x+y, x+2y  \}$. Their union is $R$ and no three form a cover, so $R$ is coverable and $\sigma(R)=4$. No quotient of $R$ is coverable, e.g., by Proposition \ref{resfield p prop} below, so $R\in S(4)$. 
\end{example}

\begin{example} \label{examplenoncomm}
Let $R = T_2(F_3)$, the ring of upper triangular $2\times 2$ matrices over the field with three elements. There are exactly four maximal subrings, given by the matrices of the following forms
$$\left\{ \begin{bmatrix}
a & \\ & b
\end{bmatrix}\right\}, 
\left\{\begin{bmatrix}
a & b\\  & a
\end{bmatrix}\right\}, 
\left\{\begin{bmatrix}
a+b & a  \\  & b
\end{bmatrix}\right\}, 
\left\{\begin{bmatrix}
a & b\\  & a+b
\end{bmatrix}\right\}.$$ There union is $R$ and no three of them form a cover, so $\sigma(R)=4$. 
The unique maximal proper quotient of $T_2(F_3)$ is $F_3\times F_3$, which is not coverable by \cite{Wer15}. Hence $T_2(F_3)\in S(4)$. 
\end{example}

Our study of $S(n)$ begins with a useful lemma. 

\begin{lemma} \label{key inequality}
Suppose $R$ is a unital ring with $\sigma(R)=n$. Let $S_1, \ldots, S_n$ be proper maximal subrings such that $R = S_1\cup \cdots \cup  S_n$. If $S:=\bigcap\limits_{i=1}^n S_i$, then $[R:S]\leq n!$
\end{lemma}

\begin{proof}
Since $S_i \not\subset \bigcup_{j\neq i} S_j$, the inequality $[R:S]\leq n!$ is a special case of Theorem 6 in \cite{Bhar09}. 
\end{proof}




\begin{Theorem} \label{Theorem 1}
The set $S(n)$ is finite. If $R\in S(n)$, then $R$ is finite 
and has characteristic a prime $p<n$. If $T$ is a unital ring then $\sigma(T)=n$ if and only if $T$ has a quotient in $S(n)$ and has no quotient in $S(m)$ for any $m<n$.
\end{Theorem}

\begin{proof}
Let $(R, S_1, \ldots, S_n)$ be a good tuple, and $S=\bigcap\limits_{i=1}^n S_i$. By Lemma 1 in \cite{Lew67} and lemma \ref{key inequality}, there is an ideal $I$ of $R$ contained in $S$ with $$[R:I] < ([R:S]+1)^{([R:S]+1)^2}  \leq (n!+1)^{ (n!+1)^2 }.$$ Since $S$ contains no nonzero ideals, $|R|< (n!+1)^{ (n!+1)^2 }$, so $R$ and $S(n)$ are finite. 

For each prime $p$, let $R(p)$ be the 2-sided ideal of $R$ which is killed by some power of $p$. Since $R$ is finite, $R = \prod\limits_{p} R(p)$. We have  $\sigma(R) = \min_p \{\sigma(R(p)) \}$ by Theorem 2.2. of \cite{Wer15}. Clearly $R$ surjects onto each $R(p)$, and $\sigma(R)<\sigma(R/I)$ for all proper ideals $I$, so $R= R(p)$ for some $p$. 

Suppose $R$ has characteristic $p^f$ and $pR\not\subset S_i$. Then by maximality of $S_i$, we have $R = S_i + pR$. So we may write $r = s + p r_1$, some $r_1\in R$. Then $r_1 = s_1 + p r_2$, and $r_2 = s_2 + p r_3$, and so on. Plugging back in and using $p^f = 0$ gives $r = s + ps_1 + p^2 s_2 + \cdots + p^{f-1}s_{f-1}$. The right side lies in $S_i$, which implies $R = S_i$. This contradiction implies $pR\subset S_i$ for all $i$, hence $pR\subset S$, so $pR=0$ and $R$ has characteristic $p$. 

Let $k=[R:S]>1$. The 2-sided ideal $kR$ is contained in $S$, hence is zero. The characteristic of $R$ is $p$, so $p$ divides $k$, which is at most $n!$, and thus $p\leq n$. In fact, $p<n$. This follows from a property of covering numbers for groups: if $G$ is a finite non-cyclic $p$-group, then the minimum number of subgroups required to cover it is at least $p+1$; see \cite{Tom97}. The abelian $p$-group $R/S$ is non-cyclic since it is covered by the proper subgroups $S_1/S, \ldots, S_n/S$. Since $R/S$ is not a union of $p$ subgroups, $R$ is not a union of $p$ subrings.

Finally, let $T$ be a unital ring with $\sigma(T)=n$. The above argument shows that $T$ has a finite quotient $T/I$ with $\sigma(T/I)=n$. If $T/I\in S(n)$ then we're done. Otherwise $T/I$ has a proper quotient with covering number $n$. Since $T/I$ is finite, we can iterate this process until we arrive at a minimal quotient with covering number $n$, and this is in $S(n)$. 
\end{proof}

We will prove one more general result about the potential commutative rings inside the sets $S(n)$. First we need a general lemma about covering numbers. 

\begin{lemma} \label{F2 times R lemma}  \label{stability of covering number under F2}
Suppose that $R_1$ and $R_2$ are finite rings with no common nonzero quotient ring. Then $\sigma(R_1\times R_2) = \min (\sigma(R_1), \sigma(R_2) )$. 
\end{lemma}

\begin{proof}
By Theorem 2.2. of \cite{Wer15}, it suffices to show that every maximal subring of $R:=R_1\times R_2$ has the form $M_1\times R_2$ or $R_1\times M_2$, where $M_i\subset R_i$ is a maximal subring. So let $T$ be a maximal subring of $R$. If we write $p_1: R \to R_1$ for the projection map onto the first factor, then clearly $T \subset p_1^{-1}(p_1(T)) = p_1(T)\times R_2$. If $p_1(T)\neq R_1$, then maximality of $T$ forces $T = p_1(T)\times R_2$, and $p_1(T)$ must be maximal since $T$ is. Thus we may assume $p_1(T)=R_1$, and similarly that $p_2(T) = R_2$, where $p_2$ is projection onto the second factor. 

Now define $I_1:=\{r\in R_1: (r, 0)\in T  \}$ and $I_2:=\{r\in R_2 : (0, r)\in T \}$. Suppose $r\in I_1$ and $r_1\in R_1$. Since $p_1(T)= R_1$, there exists some $r'\in R_2$ with $(r_1, r')\in T$. Then $T\ni (r, 0)(r_1, r') = (r r_1, 0)$ and similarly $T\ni (r_1 r, 0)$. Thus $I_1$ is a 2-sided ideal of $R_1$. If $I_1 = R_1$, then $T\supset R_1\times 0$, which implies $T = R_1\times I_2$, and then $I_2$ must be maximal since $T$ is. So we may assume that $I_1$ is a proper ideal of $R_1$, and similarly that $I_2$ is a proper ideal of $R_2$. Note that $I_1\times I_2\subset T$, and we may view $T/(I_1\times I_2)$ as a subring of $R/I_1\times R/I_2$. 

Define $f:R_1/I_1\to R_2/I_2$ by the condition $(r_1+I_1,  f(r_1 + I_1)  )\in T/(I_1\times I_2)$. Similarly define $g:R_2/I_2\to R_1/I_1$ by the condition $(g(r_2+I_2), r_2+I_2)\in  T/(I_1\times I_2)$. It is routine to check that $f$ and $g$ are well-defined ring homomorphisms and are inverse to one another, so $R_1/I_1\cong R_2/I_2$. This contradicts our initial assumption, and since all other avenues led to $T$ having the desired form, this completes the proof. 
\end{proof}

\begin{corollary}
Let $R\in S(n)$ be a commutative ring. Then all local factors of $R$ have the same residue field. 
\end{corollary}

\begin{proof}
Since $R$ is commutative and finite, it is a product of local rings. By grouping the factors in $R$, we may write $R = R_1\times \cdots R_k$ where each $R_i$ is a product of local rings with residue field $F_{q_i}$, with $q_i\neq q_j$ if $i\neq j$. Since quotients of local rings are local with the same residue field, we conclude that $R_i$ and $R_j$ have no common quotients if $i\neq j$. By Lemma \ref{F2 times R lemma}, $\sigma(R) = \min_i ( \sigma(R_i)  )$. Since $R$ surjects onto each $R_i$, the definition of $S(n)$ implies that $R=R_i$ for some index $i$. 
\end{proof}

We now consider the question of which rings are unions of four, but no fewer, proper subrings. Thus we wish to determine the set $S(4)$. By Theorem \ref{Theorem 1}, if $R\in S(4)$ then $R$ has order $p^t$ for $p\in \{2 ,3\}$ and $t\geq 1$. The rest of this paper is devoted to proving the following theorem.

\begin{theorem} \label{Theorem 2}
Let $R\in S(4)$. 

a) If $|R|=3^t $ then $R$ is isomorphic to the ring in example \ref{examplecomm} or \ref{examplenoncomm}.

b) If $|R|=2^t$, then $R$ is isomorphic to the ring in example \ref{p=2 comm} or \ref{p=2 Noncomm}. 
\end{theorem}

 \section{Covering numbers for unital rings of order $p^3$}
 
 We will compute covering numbers of all unital rings of order $p^3$, using the classification given in \cite{Antipkin}. First, we prove a lemma about rings of order $p^2$. 
 
 \begin{lemma} \label{$p^2$ coverables}
 Among unital rings of order $p^2$, only $F_2\times F_2$ is coverable. 
 \end{lemma}

 \begin{proof}
 For a prime $p$, the four unital rings of order $p^2$ are $\Z_{p^2}$, $F_{p^2}$, $F_p[t]/(t^2)$, and $F_p\times F_p$. The first of these is not coverable since it is generated by 1. The second is not coverable since it is generated by any element not in $F_p$. The third is not coverable since it is generated by $ 1+t$. 
 The last is not coverable if $p>2$ since it is generated by $(1,2)$. It is easy to see that $F_2\times F_2$ is coverable and $\sigma(F_2\times F_2)=3$.
 \end{proof}
 
 \begin{corollary}
If $R$ is a coverable unital ring of order $p^3$ that doesn't surject onto $F_2\times F_2$, then $R\in S(\sigma(R))$. 
 \end{corollary}

 \begin{lemma} \label{$p^3$ coverables}
 The following is a complete list of coverable unital rings of order $p^3$, with associated covering number $\sigma$:
 
 I) $T_2(F_p)$, the ring of upper triangular $2\times 2$ matrices over $F_p$, with $\sigma = p+1$
  
  II) $F_p[x, y]/(x^2, xy, y^2)$, with $\sigma =p+1$ 
  
  III) $F_3\times F_3\times F_3$, with $\sigma = 6$
  
  IV) $F_2\times F_2\times F_2$, with $\sigma =3$
  
  V) $F_2\times F_2[t]/(t^2)$, with $\sigma =3$
  
  VI) $F_2\times \Z_4$, with $\sigma =3$

 \end{lemma}
 \begin{proof}

 The ring $T_2(F_p)$ 
 is the unique noncommutative unital ring of order $p^3$, for any prime $p$, by \cite{Eld68}. It was shown in \cite{Wer19} that $\sigma(T_2(F_p))=p+1$. We now assume that $R$ is commutative of order $p^3$. Let $\sigma = \sigma(R)$. 
 
Suppose first that $R$ is decomposable, so $R=F_p\times T$ where $T$ has order $p^2$. If $T=F_{p^2}$, then $R$ is not coverable by Lemma \ref{F2 times R lemma}. If $T=F_p\times F_p$ then, by Theorems 3.5. and 5.3. of \cite{Wer15}, $R$ is coverable if and only if $p\leq 3$, in which case $\sigma = p +\binom{p}{2}$. If $R=F_p \times \Z_{p^2}$, then $R$ is not coverable if $p>2$ since it is generated by $ a= (1, 2)$: we have $a- a^{p(p-1)}  = (0, 1)$ and $ 2a^{p(p-1)} -a = (1,0)$. If $p=2$, this is coverable since it surjects onto $F_2\times F_2$, so $\sigma =3$. If $R=F_p\times F_p[t]/(t^2)$, it is not coverable if $p>2$ since it is generated by $a=(1, t-1)$: we have $a-a^p = (0, t)$ and $\frac{p+1}{2}(a^p + a^{2p}) = (1, 0)$, which together with $a$ give all of $R$. If $p=2 $ this is coverable since it surjects to $F_2\times F_2$, so $\sigma =3$.

Assume now that $R$ is indecomposable. If the additive group of $R$ is $C_{p^3}$, then $R$ is not coverable since no proper subring contains 1. If the additive group of $R$ is $C_p \times C_{p^2}$, then $R = \Z_{p^2} [x]/ (px, f(x))$ where $f(x) \in\{x^2, x^2 - p, x^2-kp  \}$ and $k$ is a nonsquare modulo $p$ (only if $p>2$).  In all cases, $R$ is not coverable since $\sigma(R) = \sigma(R/pR) = \sigma(F_p[x]/(x^2))=\infty$; see the proof of Theorem \ref{Theorem 1} for the first equality. 
If the additive group of $R$ is $C_p^3$ then $R$ is one of three rings. The ring $R=F_{p^3}$ is not coverable since it is generated by any element not in $F_p$. Next,  $R=F_p[x]/(x^3)$ is not coverable since it is generated by $a=1+x$: we have $a^{p^2} = 1$ and $a^{p^2} - a = x$. Last, $R=F_p[x, y]/(x^2, xy, y^2)$ is coverable with $\sigma =p+1$, by example 6.1. in \cite{Wer15}.
 \end{proof}

\begin{corollary} \label{2 choices for 27}
If $|R|=27$ and $R\in S(4)$ then $R$ is isomorphic to the ring in Example \ref{examplecomm} or \ref{examplenoncomm}. 
\end{corollary}

For future reference, we observe that if $R$ surjects onto $T_2(F_p)$ or onto $F_p[x, y]/(x^2, xy, y^2)$, and if $|R|$ is a power of $p$, then $\sigma(R)= p+1$. To see this, first note that clearly $\sigma(R)\leq p+1$. By Theorem \ref{Theorem 1}, there exists a quotient $T$ of $R$, with $T\in S(\sigma(R))$; since $T$ has order a power of $p$, Theorem \ref{Theorem 1} shows $p<\sigma(T) = \sigma(R)$. Thus $\sigma(R)=p+1$.

\section{Covering numbers of finite local rings with residue field $F_p$}

If $R$ is a finite local ring with radical $J$, then the characteristic of $R$ is $p^k$ for some prime $p$, and $J\supset pR$. So the $R$-module $J/ J^2$ is an $F_p$-vector space. 

\begin{prop} \label{resfield p prop}  
Let $R$ be a finite commutative local ring with radical $J$ and residue field $F_p$. Let  $n= \dim_{F_p} J/ J^2$. If $n\leq 1$, or if $n=2$ and $p\not\in J^2$, then $R$ is not coverable. Otherwise $R$ is coverable, surjects onto $F_p[x, y]/(x^2, xy, y^2)$, and $\sigma(R)= p+1$. In particular, if $R\in S(d)$ then $d=p+1$ and $R\cong F_p[x, y]/(x^2, xy, y^2)$. \label{ResFieldF2Corollary}

\end{prop}

\begin{proof}
 If $n=0$, then $J=J^2$ so $J=0$. A finite local commutative ring with trivial radical is a field, and fields are not coverable. 

Suppose $n\leq 2$, with $p\not \in J^2$ if $n=2$. Since $R/J = F_p$, we have $R = \Z \cdot 1_R + J$. By Nakayama's Lemma, we can lift an $F_p$-basis for $J/J^2$ to a minimal generating set for $J$. For $r\in J$, let $[r]$ denote the image of $r$ inside $J/J^2$. So we can choose $r$ so that either $\{[r]\}$ or $\{[p], [r] \}$ is a basis for $J/J^2$, depending on $n$. In either case, $J = (r, p)$ and $$R = \Z\cdot 1_R + (r).$$ Since $R$ is finite, $J$ is nilpotent and $r^m = 0$ for some $m\geq 2$. Let $a = 1_R+r$ and $t=|R^\times|$. Since $r$ is nilpotent, $a\in R^\times$, $a^t =1_R$, $a- a^t = r$, and $a$ generates $R$. 



Now assume $n\geq 2$. Let $a_1, \ldots, a_n$ be an $F_p$-basis for $J/J^2$. By Nakayama's Lemma, and since $J/J^2$ is the radical of $R/J^2$, these can be chosen so that they are also a minimal generating set for $J/J^2$ as an $R/J^2$-module. In $R/J^2$ we have the relations $a_ia_j =0$ and $pa_i = 0$ for all $1\leq i, j\leq n$. Since the characteristic of $R/J^2$ is either $p$ or $p^2$, we have shown that there is a surjective ring homomorphism 
$$\Z_{p^2}[x_1, \ldots, x_n]/I_n\to R/J^2$$ given by $x_i\mapsto a_i$, where $I_n$ is the ideal generated by the elements $px_i$ and $x_i x_j$, $1\leq i, j\leq n$. An element in the kernel, being a non-unit, is of the form $pk + \sum\limits_{i=1}^n c_i x_i$ where $0\leq k\leq p-1$ and $0\leq c_i\leq p-1$. If more than one of the $c_ix_i $ were nonzero, or if $k=0$, then we would obtain in $R/J^2$ a linear dependence among the $a_i$, which is absurd. Scaling by units, this leaves elements of the form $p+cx_i$ for $0\leq c\leq p-1$ and $1\leq i\leq n$. If two such elements (for different $c$ or different $i$) were in the kernel, we would again create a linear dependence among the $a_i$. So the kernel is a principal ideal, generated by $0$, $p$, or $p+cx_i$ for some $1\leq c\leq p-1$ and $1\leq i\leq n$. Thus $R/J^2$ is isomorphic to one of $$\Z_{p^2}[x_1, \ldots, x_n]/I_n$$ $$\Z_p[x_1, \ldots, x_n]/I_n$$ $$\Z_{p^2}[x_1, \ldots, x_{n-1}]/I_{n-1}.$$ If $n\geq 3$, all three of these surject onto $F_p[x, y]/(x^2, y^2, xy)$, so $R/J^2$ does. If $n=2$ and $p\in J^2$, then $R/J^2$ has characteristic $p$ and the third case does not arise, so we can argue as before. The surjection $R\to F_p[x,y]/(x^2, xy, y^2)$ together with the fact that $|R|$ is a power of $p$ imply the other claims immediately. 
\end{proof}

The following corollary, which will not be used in this paper, addresses a special case of a question of Werner; see \cite{Wer15}. 

\begin{corollary}
Let $R$ be a finite commutative ring, and let $R_1, \ldots, R_t$ be the local factors of $R$. Assume all $R_i$ have residue field $F_p$. 

a) If any $R_i$ is coverable, then $R$ surjects onto $F_p[x, y]/(x^2, xy, y^2)$ and $\sigma(R)=p+1$. 

b) Suppose all $R_i$ are not coverable. Then $R$ is coverable if and only if $t\geq p$, in which case $p+1\leq \sigma(R)\leq p + \binom{p}{2}$
\end{corollary}

\begin{proof}
a) This is immediate from the previous proposition. 

b) The inequality $p+1\leq \sigma(R)$ is automatic if $\sigma(R)=\infty$ and follows from Theorem \ref{Theorem 1} if $R$ is coverable. If $t\geq p$ then $R$ surjects onto $F_p^p$, which is coverable with covering number $p+\binom{p}{2}$ by \cite{Wer15}. Thus $R$ is coverable and $\sigma(R)\leq p+ \binom{p}{2}$. 

So assume $t<p$. From the proof of Proposition \ref{resfield p prop}, each $R_i$ is a quotient of $\Z_{p^{k_i}}[t_i]/(t_i^{m_i})$ for some $k_i, m_i\geq 1$. If the ring $T:=\prod\limits_{i=1}^{p-1} \Z_{p^{k_i}}[t_i]/(t_i^{m_i})$ is not coverable then neither is its quotient $R$. Let $M$ be the subring generated by  $$x = (1, 2, \ldots, p-1) + (t_1, t_2, \ldots, t_{p-1}).$$ We claim $M=T$. Let $f= |T^\times|$. 
Since each $i + t_i\in R_i^\times$, we have $x^f = (1, 1, \ldots, 1)$, so $x - x^f = (0, 1, 2, \ldots, p-2) + (t_1, t_2, \ldots, t_{p-1})$, and $(x-x^f)^f = (0, 1, 1, \ldots, 1)$. Subtracting this from $x^f$ gives $(1, 0, 0\ldots, 0)\in M$. This shows that $M\ni (1,0,\ldots, 0)(x-x^f) = (t_1, 0, \ldots, 0)$, and so $M\supset \Z_{p^{k_1}}[t_1]/(t_1^{m_1})\times 0\times \cdots \times 0$. Iterating this argument yields $M=T$, so $T$ is not coverable.
\end{proof}

\begin{prop} \label{noncomm resfield $p$ prop}
Let $R$ be a finite noncommutative local ring with residue field $F_p$. 
Then $R$ surjects onto $F_p[x, y]/(x^2, xy, y^2)$ and $\sigma(R)=p+1$. In particular, $R\not\in S(p+1)$. 
\end{prop}

\begin{proof}
Let $J$ be the radical of $R$. Since $xy-yx\in J^2$ for $x, y\in J$, and since $R = \Z\cdot 1_R + J$, we see that $R/J^2$ is commutative. Using the criteria of Proposition \ref{resfield p prop}, we will show the commutative local ring $R/J^2$ is coverable, and thus obtain a surjection $R\to R/J^2\to F_p[x, y]/(x^2, xy, y^2)$ which implies the rest. By Lemma 2.2 of \cite{Cor I}, $\dim_{F_p} J/J^2>1$. Suppose $\dim_{F_p}J/J^2 = 2$ and $p\not\in J^2$. Then there exists $x\in J$ such that $x$ and $p$ provide an $F_p$-basis for $J/J^2$. By Lemma 2.1 of \cite{Cor I}, $J = (p, x)$. But then $R = \Z\cdot 1_R + J = \Z \cdot 1_R + (x)$, so $R$ is commutative. 
\end{proof}



\section{Rings in $S(4)$ of order $3^t$}

Let $(R, S_1, S_2, S_3, S_4)$ be a good tuple, with $|R|$ a power of 3. Recall $S = \bigcap\limits_{i=1}^4 S_i$. Since $[R:S]= [R:S_i][S_i:S]\leq 4! =24$ by Lemma \ref{key inequality}, and $[R:S]$ is a power of 3, 
we have $[R:S_i]= [S_i:S]=3$ for every $1\leq i\leq 4$. In particular, $3R\subset S$, so $3R=0$. Note that $S_i\cap S_j = S$ whenever $i\neq j$, since $S_i\cap S_j$ is a proper subring of $S_i$, but $S$ is a maximal subring of $S_i$ for each $i$. Suppose there are 2 indices $i\neq j$ with $1_R\not\in S_i \cup S_j $. Then $S_i$ and $S_j$ are both maximal ideals of $R$ by Lemma 2.1. of \cite{Wer18}, and $S = S_i \cap S_j$ is an ideal, hence is zero. But $S=0$ implies $|R|=9$, contradicting Lemma \ref{$p^2$ coverables}. Thus $1_R\in S$. Let $S_1/S = \{  0, x_1, 2x_1\}$ and $S_2/S = \{0, x_2, 2x_2  \}$. We index so that $S_3/S = \{0, x_1+x_2, 2x_1+2x_2  \}$. Since $x_1+2x_2\not\in \bigcup\limits_{i=1}^3 S_i$, 
we have $S_4/S = \{ 0, x_1+ 2x_2, 2x_1 + x_2  \}$. For $r\in R\setminus S$, let $I_R(r) := \{s\in S: sr\in S  \}$ and $I_L(r) := \{s\in S: rs\in S  \}$. These are additive subgroups of $S$.


\begin{lemma}
The sets $I_R(r)$ and $I_L(r)$ are independent of $r$. 
\end{lemma}

\begin{proof}

Let $r= ax_1 + bx_2$ be one of our eight nontrivial coset representatives. We will show that the sets $I_R(r)$ are all equal. Clearly $I_R(r) = I_R(2r)$. Suppose $s\in I_R(x_1)$. If $s\not\in I_R(x_2)$ then $sx_2 = kx_2 + s_2$ for some $s_2 \in S$, $1\leq k \leq 2$. But then $s(x_1 + x_2) = sx_1 + s_2 + kx_2$ is a nontrivial coset representative of $S_2/S$, hence cannot be in $S_3$, which contradicts $S_3$ being a ring. Similarly, if $s\not\in I_R(x_1 + x_2)$, then $s(x_1 + x_2) =s_3 + kx_1 + kx_2 $ for some $s_3 \in S$, $1\leq k\leq 2$. But then $s(2x_1+x_2) = sx_1  + s_3 + kx_1 + kx_2$ lies in $S_3$ instead of $S_4$. Finally, if $s\not\in I_R(2x_1 + x_2)$, so that $s(2x_1 + x_2) = s_4 + 2kx_1 + kx_2$, for some $s_4\in S$ and $1\leq k\leq 2$, then $s x_2 = sx_1 + s_4 + 2kx_1 + kx_2$, which lies in $S_4$ instead of $S_2$. Now apply symmetry. 
\end{proof}

\begin{lemma}
The additive groups $S/I_R(r)$ and $S/I_L(r)$ have order 3. 
\end{lemma}

\begin{proof}
We can assume that $r$ is one of our eight nontrivial coset representatives. Since $1_R\in S\setminus I_R(r)$, we have $[S:I_R(r)]>1$. Suppose that $s_1, s_2, s_1 - s_2 \in S\setminus I_R(r)$. 
We claim $s_1 + s_2 \in I_R(r)$. Write $s_1r = s'_1 + k_1r$, $s_2r = s'_2 + k_2r$ with $1\leq k_i\leq 2$. Then $(s_1 - s_2)r = s'_1 - s'_2  + (k_1-k_2)r$ must have $k_1 - k_2\in \{1, 2\}$. But then $(s_1+s_2)r = s'_1 + s'_2 + (k_1 + k_2)r\in S$ since $k_1 + k_2 =0$.
\end{proof}
 

\begin{lemma}
We have $I_R(r)\cap I_L(r)=0$. 
\end{lemma}

\begin{proof}
Let $t\in I_R(r)\cap I_L(r)$. We will show that $z_1 t z_2\in S$ for $z_i$ varying among any of our eight coset representatives. Suppose first that $z_1$ and $z_2$ belong to different $S_i$. Since $z_1 t\in S$, $z_1 t z_2$ lies in the same $S_i$ as $z_2$ does. Since $tz_2\in S$, $z_1t z_2$ lies in the same $S_i$ as $z_1$ does. Since the intersection of two distinct $S_i$ is $S$, we have $z_1 t z_2 \in S$. It remains only to show that $x_1tx_1$ and $x_2tx_2$ lie in $S$. 
We have $x_1 t x_1 = s + b x_1$ where $0\leq b\leq 2$. So $x_1 t(x_1 + x_2) = s + x_1tx_2 + bx_1$. The left side lies in $S_3$, the right side in $S_1$ since $x_1tx_2 \in S$, thus $b=0$. The case $x_2tx_2$ is similar. Therefore $RtR$ is an ideal of $R$ contained in $S$. This forces $RtR=0$, so $t=0$. 
\end{proof}

Since $3\leq |S|=[S:I_R(r)\cap I_L(r)]\leq [S:I_R(r)][S:I_L(r)]= 9$, we see that $|R|=27$ or $|R|=81$. We already showed what was possible if $|R|=27$ in Corollary \ref{2 choices for 27}. The next result completes the proof of part a) of Theorem \ref{Theorem 2}

\begin{prop}  \label{prop none of 81}
If $R\in S(4)$ then $|R|\neq 81$. 
\end{prop}
\begin{proof}
Suppose $|R|=81$ and $R\in S(4)$. We have $|S|=9$ and $|S_i|=27$ for $1\leq i\leq 4$. If $R$ was commutative, then $I_R(r)$ would be a 2-sided ideal in $S$, hence would be zero. But this would force $|S|=3$, so $R$ is noncommutative. Let $J=J(R)$ denote the Jacobson radical of $R$. If $J=0$ then $R$ is semisimple, so by order considerations and noncommutativity $R=M_2(F_3)$. But $\sigma(M_2(F_3) ) = 7\neq 4$ by \cite{Wer18}. If $|J|=27$, then $R$ is a local ring with residue field $F_3$. By Proposition \ref{noncomm resfield $p$ prop}, $R\not\in S(4)$. This leaves $|J| \in \{3, 9\}$. Noncommutative unital rings of order $p^4$, for $p$ a prime, have been classified in \cite{Derr94}; we make use of their list for $R$ having characteristic $p$ and $|J| \in  \{p, p^2\}$. There are six such rings: five with $|J|=3^2$, and one with $|J|=3$. 

The unique noncommutative ring $R$ of order 81 with $|J|=3$ is $R=F_3\times T_2(F_3)$, which clearly surjects onto $T_2(F_3)$. So $\sigma(R)=4$ but $R\not\in S(4)$. 





The unique noncommutative ring $R$ of order 81 with $R/J\cong F_9$ is given by matrices of the form $$\left\{\begin{bmatrix}
a & b \\  0 & a^3
\end{bmatrix}\right\}$$ where $a, b \in F_9$. The only nonzero 2-sided ideal of $R$ is $J$, which is defined by $a=0$.  
Let $S_1$ be the index-3 subring defined by the condition $a\in F_3$. Suppose $M$ was another index-3 subring of $R$. Since $M\not\subset S_1$, $M$ contains a matrix $X$ as above with $a\in F_9\setminus F_3$. Then $[R:M]=3$ implies there exists $0\neq Y\in M\cap J$. But any such $X$ and $Y$ together generate $R$. So all other maximal subrings $S_2,\ldots , S_{\sigma(R)}$ have order at most 9. They all contain $1_R$ since they would otherwise be contained in $J$, hence in $S_1$. Thus $|S_i|=9$ for $2\leq i\leq \sigma(R)$ and $S_i\cap S_j = \{0, 1, 2\}$ for all $1\leq i\neq j\leq \sigma(R)$. Since $|R\setminus S_1| = 54$ and each other $S_i$ contributes 6 additional elements, we see that $\sigma(R)=10$.

If $R$ is noncommutative of order 81 and $R/J\cong F_3\times F_3$, then $R$ is one of the following four rings: 
$$\left\{   \begin{bmatrix}
a &  & \\ 
b & a &  \\ 
c & & d
\end{bmatrix} \right\} , 
\left\{\begin{bmatrix}
a & b & c \\ & a & \\
& & d
\end{bmatrix}  \right\},
\left\{\begin{bmatrix}
a & b & c\\ 
& d & \\
& & d
\end{bmatrix}  \right\},
\left\{\begin{bmatrix}
a & & b & \\ 
& a &  & \\
& & d & \\
& c & & d
\end{bmatrix}  \right\}$$ where $a, b, c, d\in F_3$. 
In each case, let $B$ be the matrix with $a=c=d=0$ and $b=1$. The $F_3$-span of $B$ is a two-sided ideal, and $R/(B)$ is noncommutative, so $R/(B)\cong T_2(F_3)$. Thus $\sigma(R)=4$ but $R\not\in S(4)$. 
\end{proof}

\section{Rings in $S(4)$ of order $2^t$}

Now we commence the systematic study of the 2-rings in $S(4)$, which presents more difficulties than the 3-rings. 

\begin{prop}
Suppose $(R, S_1, S_2, S_3, S_4)$ is a good tuple with $|R|$ a power of 2. Then $[R: \bigcap\limits_{i=1}^4 S_i]= 8$. 
\end{prop}

\begin{proof}
Let $S= \bigcap\limits_{i=1}^4 S_i$. Since $[R:S]=[R:S_i][S_i:S]\leq 4! = 24$ by Lemma \ref{key inequality}, we have $[R:S]\in \{ 4, 8, 16 \}$. 

We cannot have $[R:S]=4$, since the abelian group $R/S$ has the $S_i/S$ as four distinct nontrivial proper subgroups, and no group of order 4 has that many nontrivial proper subgroups.  
 
Suppose $[R:S]=16$. By Theorem \ref{Theorem 1}, $R$ has characteristic 2, so $R/S = C_2^4$. Let the subgroups $S_i/S$ be ordered so that $|S_i/S|$ is nonincreasing. Since the intersection of four index-4 subgroups contains the identity, they cannot cover the whole group. Thus $|S_1/S| =8$. Any subgroup of order 4 will intersect $S_1/S$ nontrivially, hence the remaining $S_i/S$ cannot all be order 4 (or less), since they would not be able to cover $R/S$. Thus $|S_2/S|=8$. We will show that the intersection $\bigcap\limits_{i=1}^4 S_i/S$ must be nontrivial, contradicting the definition of $S$. 
Let $w, x, y, z$ be a generating set for $C_2^4$. Without loss of generality, $S_1/S$ is generated by $\{w, x, y\}$ and $S_2/S$ by $\{w, x, z  \}$. The 4 elements not in $S_1/S\cup S_2/S$ are $y+z$, $w+y+z$, $x+y+z$, and $w+x+y+z$. A group containing any three of these will contain the fourth, so they must be split evenly in two between $S_3/S$ and $S_4/S$. If $y+z$, $w+y+z\in S_3/S$, then $w\in S/S$. If $y+z$, $x+y+z\in S_3/S$, then $x\in S/S$. And if $y+z$, $w+x+y+z\in S_3/S$, then $w+x\in S/S$. So $y+z$ cannot be paired with anything, which is a contradiction.
\end{proof}

Suppose $(R, S_1, S_2, S_3, S_4)$ is a good tuple, with $R$ a 2-ring and $S=\bigcap\limits_{i=1}^4 S_i$. For $r\in R\setminus S$, define the sets 
$$I_R(r):=\{s\in S: rs\in S  \}$$ 
$$I_L(r):=\{s\in S: sr\in S  \}$$ 
$$I_R := \bigcap\limits_{r\in R} I_R(r)$$  
$$I_L:=\bigcap\limits_{r\in R} I_L(r)$$

All four are additive subgroups of $R$. The set $I_R$ is a left $R$-ideal, and $I_L$ is a right $R$-ideal. Clearly $I_R(r) = I_R(r+s)$ for any $s\in S$, so that in studying these sets we may assume $r$ is one of the seven nontrivial elements of $R/S$. Our strategy is to bound $|S| = [S:I_R\cap I_L] |I_R\cap I_L| \leq [S:I_R][S:I_L]|I_R\cap I_L|$ by bounding $[S:I_R]$, $[S:I_L]$, and $|I_R\cap I_L|$. As a last observation, note that if $T$ is a subring of $R$ containing $S$ and $r$, then the additive map $s\mapsto rs: S\to T$ induces an injection $S/I_R(r)\to T/S$. Below we will use letters $x, y, z$ to denote elements of $R/S$ and any lift of them to $R$; the reader can check that this creates no issue. 

 Since $R$ has characteristic 2 and $[R:S]=8$, we have $R/S\cong C_2^3$. Assume the $S_i$ have been ordered so that $|S_i/S|$ is nonincreasing. If $S_1/S$ and $S_2/S$ did not both have order 4, then we could not form a cover of $R/S$. So $S_1/S =\{0, x, y, x+y  \}$ and $S_2/S = \{ 0, x, z, x+z \}$, where $x, y, z$ is some set of generators for $R/S$. 

Suppose that $S_3/S$ and $S_4/S$ also had order 4. From the inclusion-exclusion principle, $|S_i/S\cup S_j/S\cup S_k/S| = 6+ |S_i/S\cap S_j/S\cap S_k/S|$, where $i, j, k$ are distinct. So for our cover to be irredundant, we must have $S_i\cap S_j\cap S_k=S$ whenever $i, j,k$ are distinct.  We can assume, without loss of generality, that $S_3/S = \{0, y, z, y+z\}$. But now $S_4/S$, which must contain $x+y+z$, cannot contain $x$, $y$, or $z$, as this would create a nontrivial triple intersection. Similarly $S_4/S$ cannot contain $x+y$, $x+z$, or $y+z$, as any of these would, with $x+y+z$, yield one of $x$ or $y$ or $z$. So $|S_4/S|=2$. There are two options for $|S_3/S|$. 


\subsection{The case $|S_3/S| = 2$} We write $S_3/S = \{0,  y+z  \}$ and $S_4/S = \{0,  x+ y + z  \}$.  
Since $R/S_3$ and $R/S_4$ have nonprime order, the maximal subrings $S_3$ and $S_4$ contain $1_R$ by Lemma 2.1. of \cite{Wer18}. So $1_R\in S_3\cap S_4 = S$, and in particular, $|S|\geq 2$.

\begin{lemma}
We have $I_R(x) = I_R(y+z) = I_R(x+y+z) \supset I_R(y) = I_R(x+y)= I_R(z) = I_R(x+z) = I_R$. 
\end{lemma}

\begin{proof}
Suppose that $(x+y+z)s\not\in S$, so that $(x+y+z)s = x+y+z + s'$ for some $s'\in S$. Then none of $xs, ys, zs, (x+y)s, (x+z)s, (y+z)s$ can lie in $S$. For example, if $xs\in S$, then $(y+z)s = x + y + z + (s' + xs)$. The left side lies in the subring $S_3$, and the right side is in $S_4\setminus S$, which is absurd. The other cases are completely analogous, and rely only on the fact that $S_i\cap S_4=S$ for $i\neq 4$. Thus $I_R(r)\subset I_R(x+y+z)$ for every $r\in R\setminus S$.

Suppose that $(y+z)s\not\in S$, so that $(y+z)s = y+z + s'$ for some $s'\in S$. Then none of $xs, ys, zs, (x+y)s, (x+z)s, (x+y+z)s$ can lie in $S$. For example, if $xs\in S$ then $(x+y+z)s = y + z + (s'+ xs)$. The left side lies in $S_4$ and the right side lies in $S_3\setminus S$, which is absurd. The other cases are completely analogous, and rely only on the fact that $S_i\cap S_3=S$ for $i\neq 3$. Thus $I_R(r) \subset I_R(y+z)$ for every $r\in R\setminus S$. 
In particular, $I_R(x+y+z) = I_R(y+z)$.

Suppose $s\in I_R(y+z) = I_R(x+y+z)$, so that $(x+y+z)s\in S$ and $(y+z)s\in S$. We see immediately that $xs\in S$, and so $I_R(x) = I_R(y+z) = I_R(x+y+z)$. 

If $s\in I_R(y)\subset I_R(x)$ we have $ys, xs\in S$ so $(x+y)s\in S$. Symmetry implies $I_R(y)\subset I_R(x+y)$, and an identical argument shows $I_R(z)=I_R(x+z)$. Similarly, if $x\in I_R(y)\subset I_R(y+z)$ we have $ys\in S$ and $(y+z)\in S$, which immediately yields $zs\in S$, so $I_R(y)\subset I_R(z)$. By symmetry $I_R(y) = I_R(z)$. Since $I_R$ is the intersection of all $I_R(r)$, there is nothing left to show. 
\end{proof}





\begin{lemma}
We have $2\leq [S:I_R]\leq 4$ and $2\leq [S:I_L]\leq 4$. 
\end{lemma}

\begin{proof}
Since $1_R\in S \setminus I_R$, we know $S/I_R$ is nontrivial. Since $I_R=I_R(y)$, we have an injection $[s]\mapsto [ys]:S/I_R\to S_1/S$. The case for $I_L$ is symmetric. 
\end{proof}

\begin{lemma}
We have $|I_R\cap I_L|\leq 2$. 
\end{lemma}

\begin{proof}
Let $t\in I_R\cap I_L$. We have $(y+z)tx$, $(y+z)ty$, $(y+z)tz$, $xt(y+z)$, $yt(y+z)$, and $zt(y+z)$ all lying in $S$ since they lie in $S_3\cap S_i=S$ for $i=1, 2$. The same holds if we replace all occurences of $y+z$ with $x+y+z$, since $S_4\cap S_i=S$ for $i=1, 2$. From this we can conclude that $(y+z)t(y+z)\in S$ and that $ (x+y+z)t(x+y+z) = xtx + xt(y+z) + (y+z)tx + (y+z)t(y+z)\in S$. Thus $xtx\in S$. Since $(x+y+z)ty$ and $(y+z)ty$ lie in $S$, their sum $xty\in S$. Similarly $ytx\in S$. These then force $ztx, xtz\in S$.  

The remaining elements to consider are $yty, ztz, zty$, and $ytz$. From the above we see that $yty$, $zty$, $ztz$, and $ytz$ are all congruent to one another modulo $S$. So $yty\in S$ iff $RtR\subset S$ iff $t=0$. Since $ytz\in S_1\cap S_2$,  $yty\in S_1\cap S_2$. Suppose we have $0\neq a,b\in I_R\cap I_L$. If $yay, yby\not\in S$, then since $S_1/S\cap S_2/ S$ has order 2, we have $y(a+b)y\in S$. 
\end{proof}

\begin{corollary} \label{2nd Corollary}
If $(R, S_1, S_2, S_3, S_4)$ is a good tuple with $[S_1:S]= [S_2:S] = 4$ and $[S_3:S]= [S_4:S] = 2$, then $R$ has order $2^t$ with $4\leq t\leq 8$. If $R$ is commutative, then $4\leq t\leq 5$.  
\end{corollary}
\begin{proof}
We have $2\leq |S| = [S:I_R\cap I_L] |I_R\cap I_L| \leq [S:I_R][S:I_L]*2 \leq 2^5$, and so $2^4\leq |R|\leq 2^8$. If $R$ is commutative then $I_R=I_L$ is a 2-sided ideal in $S$, forcing $I_R = 0$. So  $2\leq |S| = [S:I_R]\leq 4$ and $16\leq |R|\leq 32$. 
\end{proof}

\subsection{The case $|S_3/S| =4$} 

Since $R/S_4$ has nonprime order, the maximal subring $S_4$ must contain $1_R$. In fact, $1_R$ must be contained in at least two of the other $S_i$ as well. For if, say, $1_R\not\in S_1\cup S_2$, then $S_1$ and $S_2$ would be maximal 2-sided ideals, and then the injection $R/(S_1\cap S_2) \to R/S_1\times R/S_2 \cong F_2\times F_2$ must be surjective since $[R:S_1\cap S_2]=4$, which forces $\sigma(R)=3$. But since $S_4\cap S_i=S$ for $1\leq i\leq 3$, we deduce that $1_R\in S$.   

Having written $S_1/S =\{0, x, y, x+y  \}$ and $S_2/S = \{ 0, x, z, x+z \}$, we consider what is possible for $S_3/S$ and $S_4/S$. Either $S_4/S = \{0, x+y+z\}$ or $S_4/S=\{0, y+z\}$; without loss of generality, we assume $S_4/S=\{0, x+y+z\}$. Then $S_3/S=\{0, y, z, y+z \}$ or $\{ 0, x+y, x+z, y+z  \}$, since there are 3 subgroups of order 4 containing $y+z$ and the third would contain $S_4/S$. Without loss of generality, we can take $S_3/S$ to be the former; the other choice is obtained by replacing $y$ and $z$ by $x+y$ and $x+z$, which doesn't change $S_1, S_2$, or $S_4$, nor the sizes of intersections of any number of the $S_i$. 


\begin{lemma}
We have $I_R(r)\subset I_R(x+y+z)$ for $r\in R\setminus S$. Also $I_R(x) = I_R(y+z)$, $I_R(y) = I_R(x+z)$, and $I_R(z) = I_R(x+y)$. 
\end{lemma}

\begin{proof}
Suppose $s\not\in I_R(x+y+z)$, so that $(x+y+z)s = x+y +z + s'$ for some $s'\in S$. Then we cannot have $xs$, $ys$, $zs$, $(x+y)s$, $(x+z)s$, or $(y+z)s$ in $S$. For example if we had $xs\in S$ then we would obtain $(y+z)s =x+y+z + (s' + xs) $, with the left side in $S_3$ and the right in $S_4\setminus S$, which is absurd. So $I_R(r)\subset I_R(x+y+z)$. 

Now if $s\in I_R(x)\subset I_R(x+y+z)$. Then clearly $s\in I_R(y+z)$, and by symmetry $I_R(x) = I_R(y+z)$. Similarly $I_R(y) = I_R(x+z)$ and $I_R(z) = I_R(x+y)$. 
\end{proof}

\begin{lemma}
The additive groups $S/I_R(r)$ have order 2. 
\end{lemma}
\begin{proof}
Since $1_R\in S\setminus I_R(r)$, the group $S/I_R(r)$ is nontrivial. If $r=x+y+z$, the map $[s]\mapsto [(x+y+z)s]: S/I_R(x+y+z) \to S_4/S$ is an injection. If $r=x$, the map $s\mapsto xs:S\to R$ has image in $S_1\cap S_2$. Hence it induces an injection $S/I_R(x)\to S_1/S\cap S_2/S$, and so $[S:I_R(x)]\leq 2$. Apply symmetry and the previous lemma. 
\end{proof}

\begin{lemma}
We have $[S:I_R]= [S:I_L]=2$. 
\end{lemma}
\begin{proof}
We know $I_R(r)\subset I_R(x+y+z)$, which forces $S/I_R(x+y+z)$ to be a quotient of $S/I_R(r)$. But both have order 2, so $I_R(r) = I_R(x+y+z)$, and hence $I_R = I_R(r)$. The case of $I_L$ is symmetric. 
\end{proof}

\begin{lemma}
We have $I_R\cap I_L=0$. 
\end{lemma}

\begin{proof}
Let $t\in I_R\cap I_L$. Since $z\in S_2\cap S_3$, and $x\in S_1\cap S_2$, we get $xtz\in S_1\cap S_2\cap S_3=S$. Similarly $ztx, xty, ytx, ytz, zty\in S$. We have $(x+y+z)tr$ and $rt(x+y+z)$ lying in $S$ for $r\in \{x, y, z\}$, since these lie in $S_4\cap S_i = S$ with $1\leq i\leq 3$. Thus $xtx, yty, ztz\in S$. So $RtR\subset S$ which forces $t=0$. 
\end{proof}

\begin{corollary}\label{3rd corollary}
If $(R, S_1, S_2, S_3, S_4)$ is a good tuple with $[S_1:S]=[S_2:S]=[S_3:S]= 4$ and $[S_4:S]=2$, then $R$ has order 16 or 32.  
\end{corollary}

\begin{proof}
We have $2\leq |S| = [S:I_R\cap S_L]\leq [S:I_R][S:I_L]=4$ and $[R:S]=8$. 
\end{proof}





\subsection{Rings in $S(4)$ of order 16, 32, 64, 128, and 256} 

Now we will proceed to complete the proof of Theorem \ref{Theorem 2}. The next proposition shows that $S(4)$ contains no rings of order 16 besides $F_4\times F_4$ and $M_2(F_2)$, and that $S(4)$ contains no rings of order 32. Note that noncommutative rings are coverable.

 \begin{prop} \label{noncommlema order 16} \label{lemma 16/32}  \label{noncommlema order 32}
 Let $R$ be a ring with $\sigma(R)> 3$. 
 
 a) If $R$ is commutative, coverable, and of order $16$ then $R\cong F_4\times F_4$ and $\sigma(R)=4$.
 
 b) If $R$ is commutative, coverable, and of order $32$ then $R\cong F_4\times F_4\times F_2$ and $\sigma(R)=4$. In particular, $R\not\in S(4)$. 
 
 c) If $R$ is noncommutative of order 16 then either $R=M_2(F_2)$ and $\sigma(R)=4$, or $R$ is the unique ring of order 16 with $\sigma(R)=5$.

 d) If $R$ is noncommutative of order $32$ then either $R= F_2\times M_2(F_2)$ and $\sigma(R)=4$, or $R$ is one of three rings with $\sigma(R)=5$. In particular, $R\not\in S(4)$.

 \end{prop}
 
 \begin{proof}
 a) Suppose first that $R$ is local. If $R$ has residue field $F_{16}$ then $R=F_{16}$ and is not coverable. If $R$ had residue field $F_8$ then $|R|$ would be a power of 8. By Proposition \ref{resfield p prop}, if $R$ has residue field $F_2$ then $\sigma(R)\in \{3, \infty\}$.  Last, suppose $R$ has residue field $F_4$. If $R$ has characteristic 2, then $R\cong F_4[x]/(x^2)$, which is not coverable since it is generated by $b=a+x$ for any $a\in F_4\setminus F_2$: we have $b^4 = a$ and $b- b^4=x$. If $R$ has charactersitic 4, then $\sigma(R)=\sigma(R/2R)$ so $R$ has the same covering number as a local ring of smaller order with residue field $F_4$. The only way this can happen is $R/2R = F_4$, so $R$ is not coverable.

  Now suppose that $R$ is a product of smaller local rings. If there are more than two factors, then $R = T\times F_2\times F_2$ for some ring $T$ of order 4, so $R$ surjects onto $F_2\times F_2$ and $\sigma(R)=3$. Thus $R=R_1\times R_2$, with $R_1$ and $R_2$ local. Since quotients of local rings are local, either $R_1$ and $R_2$ have the same residue field or they have no common quotients. In the latter case, Lemmas \ref{F2 times R lemma}, \ref{$p^2$ coverables}, and \ref{$p^3$ coverables} imply that $\sigma(R_1\times R_2) = \min (\sigma(R_1), \sigma(R_2  ) ) \in \{3, \infty \}$.  If $R_1$ and $R_2$ both have residue field $F_2$ then $R$ surjects onto $F_2\times F_2$ and $\sigma(R)=3$. If both have residue field $F_4$ then $R=F_4\times F_4$, and we showed in example \ref{p=2 comm} that $R$ is coverable with $\sigma(R)=4$.

  b) Suppose $R$ is local. If $R$ has residue field $F_{32}$ then $R=F_{32}$ and is not coverable. Since 32 is not a power of 16, 8, or 4, $R$ does not have residue field $F_{16}$, $F_{8}$ or $F_4$.  By Proposition \ref{resfield p prop}, if $R$ has residue field $F_2$ then either $R$ is not coverable or $\sigma(R)=3$.

  Now let $R$ be a product of local rings, $R=\prod\limits_{i=1}^t R_i$, with $2\leq t\leq 5$. If $t\geq3$, then either two (or more) factors have residue field $F_2$, so $\sigma(R)=3$, or else $R\cong F_4\times F_4\times F_2$, which has covering number 4 by Lemma \ref{F2 times R lemma}. This just leaves the case $R=R_1\times R_2$ with $R_1$ and $R_2$ local. Suppose $|R_1|=16$, so $R_2 = F_2$. If $R_1$ has residue field $F_2$, $R$ surjects onto $F_2\times F_2$ and $\sigma(R)=3$. Otherwise $R_1=F_{16}$ or $R_1$ has residue field $F_4$. We already showed that none of these are coverable, so by Lemma \ref{F2 times R lemma}, $R_1\times F_2$ is not coverable. Finally, suppose $|R_1|=8$ and $|R_2|=4$. If they both have residue field $F_2$ then $R$ surjects onto $F_2\times F_2$ and $\sigma(R)=3$. Otherwise they must have different residue fields, hence no common quotients, and by Lemmas \ref{F2 times R lemma}, \ref{$p^2$ coverables}, and \ref{$p^3$ coverables} we have $\sigma(R_1\times R_2) = \min (\sigma(R_1), \sigma(R_2  ) ) \in \{ 3, \infty  \}$.

 c) We consider $R$ in cases according to the size of its radical $J=J(R)$. Since $R$ has a unit, $|J| \in \{0, 2, 4, 8  \}$. If $J = 0$ then $R$ is a semisimple ring, so by order considerations and noncommutativity $R=M_2(F_2)$. We showed that $\sigma(M_2(F_2) ) =4$ in example \ref{p=2 Noncomm}. Suppose $|J|=2$. If $R/J = F_2^3$, then $R$ surjects onto $F_2\times F_2$, giving $\sigma(R)=3$. We cannot have $R/J=F_8$ or $R/J = F_2\times F_4$; see \cite{Derr94}. Suppose $|J|=4$. If $R/J=F_2\times F_2$, then $\sigma(R)=3$. The only other possibility is $R/J = F_4$. By \cite{Derr94} or \cite{Cor I}, there is a unique such $R$, isomorphic to the ring of matrices of the form $$\left\{  \begin{bmatrix}
 a & b \\  0 & a^2
 \end{bmatrix}\right\}  $$ where $a, b \in F_4$. The proof that $\sigma(R)=5$ is completely analogous to the proof in Proposition \ref{prop none of 81} that the corresponding ring of order 81 had covering number 10, so we omit the details. 
 Finally, if $|J|=8$ then $\sigma(R)=3$ by Proposition \ref{noncomm resfield $p$ prop}. 
 
 d) Suppose $R$ is local. Since $R\neq F_{32}$, the residue field of $R$ is $F_2$. Then $\sigma(R)=3$ by Proposition \ref{noncomm resfield $p$ prop}. 
 
 If $R$ is decomposable, then at least one of its summands is noncommutative. Unital rings of order less than $8$ are commutative. The only noncommutative ring of order 8 is $T_2(F_2)$, which has covering number 3, and so $\sigma(R)=3$ if $T_2(F_2)$ is a summand of $R$. The remaining case is $R=F_2\times R'$ for some noncommutative ring of order $16$. If $\sigma(R')=3$ then $\sigma(R)=3$. If $\sigma(R')=4$ then $R'=M_2(F_2)$ by part b).  By part b) again, the only remaining possibility is that $R'$ is the unique noncommutative ring of order 16 and covering number 5. Since $R'$ does not surject onto $F_2$, Lemma \ref{F2 times R lemma} implies $\sigma(F_2\times R') = \sigma(R')=5$.

 Finally, suppose that $R$ is indecomposable and nonlocal. From the classification in \cite{Cor I}, there are exactly 2 nonlocal, indecomposable rings of order $32$ that do not surject onto $F_2\times F_2$; they are opposite rings of one another, so they have the same covering number.  
  Concretely, we may realize one of them as the ring of matrices $$\left\{ \begin{bmatrix}
   a & b \\0 & c
   \end{bmatrix} \right\}$$ where $a\in F_2$, $b, c\in F_4$. 
  The radical $J$ of $R$ consists of elements with $a=c=0$. If $x\in R$ is a matrix with $a=1$, arbitrary $b$, and $c\not\in F_2$, and $0\neq y\in J$, then $x$ and $y$ generate $R$. Thus no proper subring of index $2$ can contain such an element $x$, since any subring of index 2 intersects $J$ nontrivially. Since $x^3 = 1_R$, the subring $S_x$ generated by $x$ consists of elements of the form $i_1 1_R + i_2x + i_3 x^2$, where $i_j\in F_2$. It is easy to check that these are 8 distinct elements in $R$, so $[R:S_x]=4$, and $S_x$ is therefore a maximal subring generated by a single element. Thus $S_x$ must be included in any cover of $R$. Since no element of $S_x$ besides $x$ and $x^2$ have $a=1$ and $c\not\in F_2$, we see that $S_x = S_{x'}$ if and only if $x'\in  \{x, x^2\}$. There are 8 elements with $a=1$ and $c\not\in F_2$, so there are 4 distinct subrings of the form $S_x$.  The four $S_x$ cannot cover $R$, being of index 4 and having nonempty intersection, so $\sigma(R)>4$. Let $T$ be the index-2 subring consisting of matrices with $c\in F_2$. One sees that $S_x\setminus (T\cup S_{x'}) = \{ x, x^2, 1_R + x, 1_R+x^2\}$ for any $x'\neq x, x^2$. Since $|R\setminus T|=16$ and there are four elements in each of the four $S_x$ that lie in no other of our maximal proper subrings, we have $R = T\cup (\bigcup S_x)$ and $\sigma(R)=5$. 
 \end{proof}

From Proposition \ref{noncommlema order 16}, Corollary \ref{2nd Corollary}, and Corollary \ref{3rd corollary}, we now know that if $(R, S_1, \ldots, S_4)$ is a good tuple and $R$ is not $F_4\times F_4$ or $M_2(F_2)$, then $R$ is noncommutative, $|R|\in \{64, 128, 256\}$, $[R:S_1]=[R:S_2]=2$, $[R:S_3]=[R:S_4]=4$, $[R:S]=8$, and $1_R\in S$. We also saw that $S = S_i\cap S_j$ if $i\neq j$ except that $[S_1\cap S_2:S]=2$. In particular, $S_i\cap S_j\cap S_k = S$ for any three distinct indices $i,j,k$. 

{\it For the remainder of this section, $R$ will denote a hypothetical ring in $S(4)$ satisfying these conditions.} We will show that no such ring exists, thereby completing the proof of Theorem \ref{Theorem 2}. 

\begin{prop}
$R$ is indecomposable. 
\end{prop}

\begin{proof}
Suppose that $R$ is decomposable. If $R = R_1\times R_2$ where $R_1$ and $R_2$ have no common quotients. By Lemma \ref{F2 times R lemma}, we then have $\sigma(R) = \min ( \sigma(R_1), \sigma(R_2 ) )$. So $R\not\in S(4)$ because $R$ has a quotient with the same covering number, contradicting the definition of $S(n)$. In any decomposition of $R\in S(4)$ into a sum of indecomposables, there must then be pairwise common quotients. 

If $R = R_1\times R_2$ where $R_1$ is of order 2 or 4, then the minimal quotient of $R_1$ is either $F_2$ or $F_4$, so $R_2$ will have to surject onto $F_2$ or $F_4$ for there to be a common quotient. In the former case $R$ surjects onto $F_2\times F_2$ and $\sigma(R)=3$, while in the latter case $R$ surjects onto $F_4\times F_4$ so $R\not\in S(4)$. Summands of $R$ therefore have order at least 8. 

Clearly $R$ must have a noncommutative indecomposable summand. Since $T_2(F_2)$ is the only noncommutative ring of order 8 and $\sigma(T_2(F_2))=3$, $R$ has a noncommutative summand of order at least 16. Write $R'$ for the unique ring of order 16 and covering number 5. Suppose $R = T\times R'$, where $T$ is any ring of order $|R|/16$. Since $F_4$ is the only proper quotient of $R'$, $F_4$ must be a quotient of $T$. But then $R$ surjects onto $F_4\times F_4$, so $R\not\in S(4)$. By Proposition \ref{noncommlema order 16} no noncommutative ring of order 16 other than $R'$ can be a summand of $R$. Thus $R$ must have a noncommutative indecomposable summand of order at least 32.

These observations rule out $|R|=64$ and $|R|=128$. Let $R = R_1\times R_2$ where $|R_1|=8$, $|R_2|=32$, and $R_2$ is noncommutative and indecomposable. If $R_1$ and $R_2$ have a common quotient of $F_2$ or $F_4$, then we can argue as before to deduce $R\not\in S(4)$. The only other apparent possibility would be a common quotient of $F_8$. However, this would contradict $R_2$ being noncommutative and indecomposable; see the classification in \cite{Cor I}, for example. 
\end{proof}

\begin{lemma} \label{S3cap I = Scap I lemma}
 $R$ contains no 2-sided ideals $I$ with $0\neq S\cap I = S_3\cap I = S_4\cap I $.
\end{lemma}
\begin{proof}
If $I$ is such an ideal in $R$ then $S_i\cap I$ is an ideal in $S_i$. Therefore $S\cap I$ is an ideal in both $S_3$ and in $S_4$, hence is an ideal in the ring generated by $S_3$ and $S_4$. Maximality of the $S_i$ then implies that $S\cap I$ is an ideal of $R$, contained in $S$. So $\sigma(R) = \sigma(R/(S\cap I))$, contradicting $R\in S(4)$. 
\end{proof}

\begin{prop}
$R$ is not a local ring. 
\end{prop}

\begin{proof}
By Proposition \ref{noncomm resfield $p$ prop}, $R$ cannot have residue field $F_2$. Thus $|R|\neq 2^7$ since $R\neq F_{128}$. Let $J=J(R)$ be the radical of $R$, and $f:R\to R/J$ the quotient map. 

{\bf Case 1a:} $|R|=64$ and $R/J = F_8$. The relevant subrings have orders $|S_1|=|S_2|=32$,  $|S_3|=|S_4|=16$, and $|S|=8$. Since $1_R\in S\setminus J$, we have $|S\cap J|\in \{0, 2, 4\}$.  

If $S\cap J = 0$ then $f|_S:S\to R/J$ is injective, hence surjective. If $0\neq j\in J$ then $\{sj\}_{s\in S} = J$, since $J=J/J^2$ is a 1-dimensional vector space over $F_8$. The subring generated by $S$ and $j$ then contains $J$, and $R= S+J$. But this means $j\not\in S_i$ for any $i$, contradicting the definition of a covering. 

If $S\cap J$ had order 2 then $f(S)=S/S\cap J$ would have order 4 and be a subring of $R/J = F_8$. This is impossible. 

Suppose $S\cap J$ has order 4. The maps $f|_{S_i}:S_i\to R/J$ contain $S\cap J$ in their kernels. So $f(S_3)$ and $f(S_4)$ have order at most 4, hence at most 2 since $F_8$ contains no subring of order 4. But $f^{-1}( F_2)$ is then a proper subring of $R$ containing the maximal subrings $S_3$ and $S_4$, a contradiction.

{\bf Case 1b:} $|R|=64$ and $R/J = F_4$. The sizes of the subrings $S_i$ and $S$ are as before. Since $1_R\in S\setminus J$, and $S\cap J $ is the kernel of the map $f|_S:S\to R/J=F_4$, we have $|S\cap J|\in \{2, 4\}$.

If  $S\cap J$ has order 2 then $f|_S:S\to R/J=F_4$ is onto. Since $J^3 =0 $ we have $$|S\cap J| = \prod\limits_{i=1}^2|(S\cap J^i)/(S\cap J^{i+1})|.$$ Observe that $(S\cap J^i)/(S\cap J^{i+1})$ is an $F_4$-subspace of $J^i/J^{i+1}$ since it is stable under multiplication by $S/(S\cap J) = F_4$. So $|(S\cap J^i)/(S\cap J^{i+1})|$ is a power of $4$, but their product equals $|S\cap J| =2$. 

If $S\cap J$ has order 4, then $f(S)=S/(S\cap J)= F_2$. The preimage $f^{-1}(F_2)$ is a subring of index 2, so it cannot contain the maximal subrings $S_3$ or $S_4$, which have index 4. Therefore $f(S_3) = f(S_4) = R/J$, and it follows that $S_3\cap J$ and $S_4\cap J$ have order 4, being the kernels of $f|_{S_i}:S_i\to R/J = F_4$. But then the inclusion $S\cap J \subset S_i\cap J$ and the assumption that $S\cap J$ has order 4 means that $S\cap J = S_i \cap J$ for $i=3, 4$. Now apply Lemma \ref{S3cap I = Scap I lemma} 

{\bf Case 2a:} $|R|=256$ and $R/J = F_{16}$. We have $|S_1|=|S_2|=128$, $|S_3|=|S_4|=64$, and $|S|=32$. Note that $J$, which has order 16, is not contained in $S$ because $S$ contains no 2-sided ideals of $R$. Thus $|S\cap J|\in \{2, 4, 8\}$. 

If $S\cap J$ has order 2 then $f|_S :S\to R/J$ is onto. Noting that $J^2=0$, we see that $S\cap J$ is a vector space over $S/(S\cap J) = f(S)= F_{16}$. But 16 does not divide $|S\cap J|=2$.   

If $S\cap J$ has order 4 then $f(S) = S/(S\cap J)$ has order 8 and is a subring of $R/J = F_{16}$. This is impossible. 

Finally, suppose $S\cap J$ has order 8. The maps $f|_{S_i}:S_i\to R/J$ contain $S\cap J$ in their kernels. In particular, $f(S_3)$ and $f(S_4)$ have order at most $|S_3/(S\cap J)|= |S_4/(S\cap J)| = 8$. Since $F_{16}$ contains no subring of order 8, both $S_3$ and $S_4$ are contained in $f^{-1}(F_4)$, since $F_4$ is the unique subring of $F_{16}$ of order 4. 
This contradicts the maximality of $S_3$ and $S_4$.

{\bf Case 2b:} $|R|=256$ and $R/J = F_4$. The sizes of the subrings $S_i$ and $S$ are as before. Since $f(S) = S/(S\cap J)$ injects into $R/J$ and $1_R\in S\setminus J$, we have $|S\cap J|\in \{ 8, 16  \}$.

If $S\cap J$ has order 8, then $f|_S:S\to R/J$ is onto. Since $J^4 =0 $ we may write $$|S\cap J| = \prod\limits_{i=1}^3|(S\cap J^i)/(S\cap J^{i+1})|.$$ Observe that $(S\cap J^i)/(S\cap J^{i+1})$ is an $F_4$-subspace of $J^i/J^{i+1}$ since it is stable under multiplication by $S/(S\cap J) =f(S)=  F_4$. So $|(S\cap J^i)/(S\cap J^{i+1})|$ is a power of $4$, but their product equals $|S\cap J| =8$. 


If $S\cap J$ has order 16, then $f(S)=F_2$. Since $f^{-1}(F_2)$ is an index-2 subring of $R$, it cannot contain the index-4 maximal subrings $S_3$ or $S_4$. So $f(S_3)=f(S_4) = R/J$. Thus $S_3\cap J$ and $S_4\cap J$ have order 16, as kernels of the maps $f|_{S_i}:S_i\to R/J = F_4$. But then the inclusions $S\cap J \subset S_i\cap J$ and the assumption $|S\cap J|=16$ imply that $S\cap J = S_i \cap J$ for $i=3, 4$. Now apply Lemma \ref{S3cap I = Scap I lemma}. 
\end{proof}

By the previous two propositions, $R$ is indecomposable and nonlocal. Following the ideas in \cite{Cor I}, there exist orthogonal idempotents $e_1, e_2, \ldots, e_m\in R$ such that $\sum\limits_{i=1}^m e_i =1$, $R_i:= e_i R e_i$ is a full matrix ring over a local ring, and $M_{ij} := e_i R e_j$ is an $(R_i, R_j)$-bimodule contained in the radical $J$ of $R$. Let $M = \bigoplus\limits_{i\neq j} M_{ij}$. Then $J = J(R_1)\oplus \cdots \oplus J(R_m)\oplus M$. In particular, if $R/J = \bigoplus\limits_{i=1}^m R_i/J(R_i)$ has covering number 3 or 4, then $R\not\in S(4)$. Since $|R|\leq 256$, the only way some $R_i$ could be nonlocal is if $R_i = M_2(T)$, where $T$ is a local ring of order 2 or 4. If $T$ surjects onto $F_2$ then $R_i/J(R_i) = M_2(F_2)$ which has covering number 4, and $R\not\in S(4)$. Otherwise, $R_i = R =M_2(F_4)$, as this has order $256$. But $\sigma(M_2(F_4))>4$ by \cite{Wer18}. {\it Thus the subrings $R_i$ are all local rings}. 

We now study some constraints on $M$ and $R$, and then proceed to eliminate each possible value of $|R|$.

\begin{lemma} \label{indecomp lemma}
Write $R = \left(\bigoplus\limits_{i=1}^m R_i\right) \oplus M$ as above. The following hold. 

a) $|M_{ij}|\geq 4$ for all $1\leq i\neq  j\leq m$. In particular, $|M|\geq 4$. 

b) $ 8\leq [R:M]\leq 64$. 

c) If $[R:M] = 8$ then $m=2$, $\{R_1, R_2\} = \{F_2, F_4\}$, and $|R|\not\in \{64, 256\}$. 

d) If $[R:M]=16$ then $m=2$. If $|R|\in \{64, 256\}$ then $\{R_1, R_2\} = \{ F_2[x]/(x^2) , F_4 \}$. 

e) If $[R:M]=32$ or $64$, then $m=2$. 
\end{lemma}

\begin{proof}
 a) Since $R$ is indecomposable and nonlocal, we have $m\geq 2$ and $M\neq 0$. If some $M_{ij}$ has order 2 then the homomorphisms $R_i\to \End_\Z(M_{ij}) = F_2$ and $R_j\to \End_\Z(M_{ij})=F_2$ coming from the bimodule structure would force $R_i$ and $R_j$ to have residue field $F_2$, from which we get a surjection $R\to R/J \to  F_2\times F_2$, and $\sigma(R)=3$. 
 
 b) If $[R:M]=4$ then $m=2$ and $R_1=R_2 = F_2$, so we get a surjection $R\to R/J \to F_2\times F_2$, and $\sigma(R)=3$. Thus  $[R:M]\geq 8$. The inequality $[R:M]\leq 64$ is clear since $|M|\geq 4$ and $|R|\leq 256$.   
 
 c) Suppose $[R:M]=8$. If $m=3$ then $R_1=R_2=R_3=F_2$, which forces $\sigma(R)=3$. If $m=2$ then up to indexing, $R_1 = F_2$ and $R_2$ is a local ring of order 4. If $R_2$ has residue field $F_2$ then once again $\sigma(R)=3$, so $R_2 = F_4$. Thus $|M|$ is a power of $4$, which together with $[R:M]=8$ rules out $|R|\in \{64, 256\}$. 
 
 d) Suppose $[R:M] = 16$. Then $m=2$ since $m\geq 3$ would imply at least two $R_i$ being $F_2$. If, say, $R_1 = F_2$ then $R_2$ is a local ring of order 8 that does not surject onto $F_2$, so $R_2=F_8$. This forces $|M|$ to be a power of 8, which rules out $|R|\in \{64, 256\}$. If neither $R_1$ nor $R_2$ is $F_2$, then both have order 4, and since $R$ does not surject onto $F_2\times F_2$ or $F_4\times F_4$, up to indexing we have $R_1 = F_4$ and $R_2\in \{\Z_4, F_2[x]/(x^2)  \}$. Since $R$ has characteristic 2, $R_2=F_2[x]/(x^2)$.  
 
 e) Suppose $[R:M]=32$. Note this cannot happen if $|R|=64$. If $m\geq 4$ then at least 2 of the $R_i$ equal $F_2$, and $\sigma(R)=3$. If $m=3$, then at least one $R_i=F_2$. If more than one is $F_2$ then again $\sigma(R)=3$. Otherwise, either the other two local rings $R_i$ both equal $F_4$, or at least one of them surjects onto $F_2$. In either case $\sigma(R/J)\in \{3, 4\}$ and $R\not\in S(4)$. So $m=2$. 
 
 Finally, suppose $[R:M]=64$, which is only possible if $|R|=256$. It is easy to see that if $m\geq 4$ then either two $R_i$ have residue field $F_2$ or equal $F_4$, and both of these cases force $R\not\in S(4)$. If $m=3$, and no two of $R_1, R_2$, and $R_3$ have common residue field $F_2$ or $F_4$, then $R_1 = F_2$, $R_2 = F_4$, and $R_3 = F_8$, up to indexing. But $|M| = 4$ so $M$ is not vector space over $F_8$. 
\end{proof}

In particular, we see that $m=2$ in all cases. The following lemma will rule out several cases below. Note that $M_{ij}M_{ji}\subset J(R_i)$, so in general $M$ is not a 2-sided ideal of $R$, but it is if $M^2=0$. 

\begin{lemma} \label{R/J = F2 times F4 impossibility}
The following cannot occur: $R = R_i \oplus R_j \oplus M$, with $R_i$ and $R_j$ having residue fields $F_2$ and $F_4$, respectively, and $M^2=J(R_j)M = MJ(R_j)=0$. 
\end{lemma}

\begin{proof}
Let $f:R\to R/J=F_2\times F_4$ be the quotient map. Suppose $T$ is a subring of $R$ with $[R:T]=2$ and $1_R\in T$. Since $[f(R): f(T)]\leq [R:T]=2$, and $f(1_R) = (1,1)\in f(T)$, either $f(T)=F_2\times F_2$ or $f(T)=F_2\times F_4$. 

If $f(T) = F_2\times F_2$ then $T\subset f^{-1}(F_2\times F_2)\subsetneq R$, and maximality of $T$ implies $T = f^{-1}(F_2\times F_2)$. In particular, $T\supset J \supset M$. 

Suppose $f(T) = F_2\times F_4$. Let $t\in T$ such that $f(t)=(0, c)\in F_2\times F_4$ where $c\not\in F_2$. Since $[R_j:T\cap R_j]\leq 2$ and $[R_j: J(R_j)]>2$, we have $T\cap R_j\not\subset J(R_j)$. Therefore $T$ intersects the unit group $R_j^\times = R_j\setminus J(R_j)$. Finiteness then implies $T\ni e_j$, the idempotent of $R$ giving the unit in $R_j= e_jR e_j$. Then  $e_j t e_j\in T\cap R_j$ and $f(e_j t e_j ) = (0,c)$. It follows that $T\cap M$ is an $F_4$-subspace of $M$, since $T$ contains a set of representatives for $R_j/J(R_j)=F_4$. But $[M:T\cap M]\leq 2$, so we must have $T\supset M$. 

Now, $1_R\in S_1\cap S_2$ and $[R:S_1]=[R:S_2]=2$, so the argument just given shows that $S_1\cap S_2\supset M$. Since $[S_1\cap S_2: S] = 2$, we have $[M:S\cap M]\leq 2$. We cannot have $S\supset M$, because the assumption $M^2 =0$ means $M$ is a 2-sided ideal of $R$. 
Since $S= S_1\cap S_2\cap S_3$, we have $S_3\not\supset M$ and $2\leq [M:S_3\cap M]\leq [M: S\cap M]\leq 2$. The inclusion $S\cap M \subset S_3\cap M$ now forces $S\cap M = S_3\cap M$, and similarly $S\cap M = S_4\cap M$. Now apply Lemma \ref{S3cap I = Scap I lemma}. 
\end{proof}


\begin{prop}
$|R|\neq 64$.
\end{prop}

\begin{proof}
Suppose $R$ has order 64. From Lemma \ref{indecomp lemma} and its proof, $[R:M]=16$, $m=2$, $\{R_1, R_2\} = \{ F_4, F_2[x]/(x^2) \}$ and $M=M_{12}$ has order 4.  
Because the $e_i$ are orthogonal idempotents, $M^2=M_{12}^2= 0$. 
Now apply Lemma \ref{R/J = F2 times F4 impossibility}.  
\end{proof}

\begin{prop}
$|R|\neq 128$. 
\end{prop}

\begin{proof}
Suppose $R$ has order 128. By Lemma \ref{indecomp lemma} and its proof, there are three possible values for $[R:M]$, all of which have $m=2$. 

{\bf Case 1:} $[R:M]=8$, so $|M|=16$. Lemma \ref{indecomp lemma} shows $\{R_1, R_2\} = \{F_2, F_4\} $. In particular, $J= M$ is a 2-dimensional vector space over $F_4$. Then $M^2 = M_{12}M_{21}\oplus M_{21}M_{12}\subset J(R_1)\oplus J(R_2)=0$. Now apply Lemma \ref{R/J = F2 times F4 impossibility}.

 {\bf Case 2:} $[R:M]=16$, so $|M|=8$. If $R_1\neq F_2$ and $R_2\neq F_2$, then $\{R_1, R_2\} = \{F_4, F_2[x]/(x^2)\}$. But $M$ has order 8 so is not an $F_4$-vector space. So up to indexing, $R_1 = F_2$ and $R_2 = F_8$. Thus $M=J$ is a 1-dimensional $F_8$-vector space.
 As $R_2= F_8$ has no index-2 subrings and $[R_2: S_1\cap R_2] \leq 2$, we have $S_1\supset R_2$. Since $|S_1||J|>|R|$ there exists $0\neq j\in S_1\cap J$.  Then $R_2$ and $j$ generate the 1-dimensional $F_8$-vector space $J$, so $S_1\supset  J$. Therefore $S_1/J\subset R/J = F_2\times F_8$ has order 8 and contains $(1, 1)$. This is impossible.

{\bf Case 3:} $[R:M]=32$, so $|M|=4$. By Lemma \ref{indecomp lemma} either $M=M_{12}$ or $M=M_{21}$. In particular, $M^2 =0$. 

If $R_1 = F_2$, then $|R_2|=16$. Since $|M|<16$,  $R_2\neq F_{16}$, and $R_2$ does not have residue field $F_2$ since that would give a surjection $R\to F_2\times F_2$. So $R_2$ has residue field $F_4$. 
Suppose $M= M_{12}$. By Nakayama's lemma, $M/MJ(R_2)$ is a nonzero (right) module over $R_2/J(R_2)  = F_4$, so has order at least 4. Thus $MJ(R_2)=0$. Similarly, if $R_2M \neq 0$ then $J(R_2)M=0$. Now apply Lemma \ref{R/J = F2 times F4 impossibility}. The case $M=M_{21}$ is similar.

If $R_1\neq F_2$, we can assume $R_1$ has order 4 and $R_2$ has order 8. Since $|M|<8$, we conclude that $R_2$ has residue field $F_2$, which forces $R_1=F_4$. 
Since $MJ(R_2)$ is an $F_4$-subspace of $M$ and $MJ(R_2)\neq M$ by Nakayama's Lemma, we have $MJ(R_2)=0$. Similarly $J(R_2)M=0$, so apply Lemma \ref{R/J = F2 times F4 impossibility}. 
\end{proof}

\begin{prop}
$|R|\neq 256$. 
\end{prop}

\begin{proof}
Suppose $R$ has order 256. From Lemma \ref{indecomp lemma} and its proof, there are three possible values for $[R:M]$, all of which have $m=2$. 

{\bf Case 1:} $[R:M]=16$, so $|M|=16$. By Lemma \ref{indecomp lemma}, $R_1=F_4$ and $R_2= F_2[x]/(x^2) $, up to indexing. 

Suppose $M^2 \neq0$. Then $M^2 = M_{12}M_{21}\oplus M_{21}M_{12}= J(R_2)$ has order 2. It follows that any subring of $R$ containing $M$ also contains $J=J(R_2)\oplus M$. The proof of Lemma \ref{R/J = F2 times F4 impossibility} shows that if $T$ is an index-2 subring of $R$ and  $T\ni 1_R$ then $T\supset M$. We note that this did not assume $M^2 =0$. Therefore $S_1\supset J$ and $S_2\supset J$. But then the image of $S_1$ and $S_2$ in $R/J = F_2\times F_4$ is $F_2\times F_2$, which means $S_1\cup S_2$ is contained in the preimage of $F_2\times F_2$, a contradiction. Therefore $M^2=0$ and Lemma \ref{R/J = F2 times F4 impossibility} applies. 


{\bf Case 2:}  $[R:M]=32$, so $|M|=8$. By Lemma \ref{indecomp lemma}, $M$ has only one summand $M_{ij}$ of order 8, else $M$ would have a summand of order 2.  In particular, $M^2 =0$. 

If $R_1=F_2$ then $R_2$ is a local ring of order 16 and with residue field not equal to $F_2$. Since $|M|<16$, the residue field of $R_2$ must be $F_4$. We have $|J(R_2)|=4$ and $J(R_2)^2 =0$. Assume $M = M_{12}$; the case $M=M_{21}$ is similar and corresponds to the opposite ring. 
By Nakayama's Lemma, $M/MJ(R_2)$ is a nonzero right module over $R_2/J(R_2)=F_4$. Since $M$ has order 8, this prevents $MJ(R_2)$ from being zero. So $MJ(R_2) = MJ(R_2)/MJ(R_2)^2$ is also a nonzero right module over $R_2/J(R_2)$. But then both $M/MJ(R_2)$ and $MJ(R_2)$ have orders divisible by 4, and with product 8.

If $R_1\neq F_2$ then, up to indexing, $R_1$ has order 4 and $R_2$ order 8. Since $|M|$ is not a power of $4$, we cannot have $R_1= F_4$. So $R_1$ must have residue field $F_2$, and since $R$ has characteristic 2, this forces $R_1 = F_2[x]/(x^2)$. The local ring $R_2$ cannot also have residue field $F_2$, so $R_2=F_8$. We claim that there is exactly one subring of $R$ of index 2 that contains $1_R$. This will contradict the existence of $S_1$ and $S_2$. 
Towards that end, let $T$ be an index-2 subring of $R$ containing $1_R$. Then $[R_2:T\cap R_2]\leq 2$, which forces $T\supset R_2$ since $R_2 = F_8$ has no subrings of index 2. Similarly, there exists $0\neq m\in T\cap M$ since $[M:T\cap M]\leq 2$. But $M$ is a 1-dimensional vector space over $R_2 = F_8\subset T$, so $T\supset M$. Since $1_R = e_1 + e_2\in T$ and $e_2\in T$, we have  $T \supset F_2 \oplus R_2 \oplus M$. The right side being an additive subset of $R$ of index $2$, we conclude $T = F_2\oplus R_2\oplus M$.

{\bf Case 3:} $[R:M]=64$, so $|M|=4$. By Lemma \ref{indecomp lemma}, $M$ has only one summand, so $M^2 =0$. 
If $R_1=F_2$, then $R_2$ is a local ring of order 32. Such a ring either has residue field $F_2$, in which case $\sigma(R)=3$, or else is equal to $F_{32}$, which is impossible since $|M|<32$. If $R_1$ and $R_2$ have order 8, then either they both surject onto $F_2$ and $\sigma(R)=3$, or else at least one of them is $F_8$. But this is again impossible since $|M|<8$. Thus $R_1$ has order 4 and $R_2$ order 16. The residue field of $R_2$ cannot be $F_{16}$ because $|M|=4$. If $R_1$ and $R_2$ both have the same residue field $F_2$ or $F_4$, then $R\not\in S(4)$. So one has residue field $F_2$ and the other $F_4$. Since $R$ has characteristic 2, the two choices for $R_1$ are $F_4$ and $F_2[x]/(x^2)$.

 If $R_1=F_4$ and $R_2$ has order 16 with residue field $F_2$, Lemma \ref{R/J = F2 times F4 impossibility} applies. 

Suppose $R_1 = F_2[x]/(x^2)$ and $R_2$ has order 16 and residue field $F_4$. 
If $M= M_{12}$, then $M/MJ(R_2)$ is nonzero by Nakayama's Lemma and is a module for $R_2/J(R_2) = F_4$. This and $|M|=4$ forces $MJ(R_2)=0$. The case $M=M_{21}$ is similar. Now apply Lemma \ref{R/J = F2 times F4 impossibility}. 
\end{proof}

This completes the proof of part b) of Theorem \ref{Theorem 2}.

\section{Further directions}


This paper established a strategy to classify unital rings of covering number $n$ that requires a finite amount of computation. It is possible that $S(n)$ is still finite if non-unital rings are permitted. This is known if $n=3$; see \cite{Maro12}. Perhaps the more general case requires no new ideas from those presented here. However, the actual computation of $S(n)$ would likely require different techniques. 

Propositions \ref{resfield p prop} and \ref{noncomm resfield $p$ prop} applied only to local rings with residue field $F_p$. Having a similar criterion for local rings with arbitrary residue field would help significantly in classifying which commutative rings are coverable, and in any attempt to compute $S(n)$ for larger $n$. For example, $S(5)$ contains both a noncommutative and commutative local ring with residue field $F_4$. 

It is unknown whether $S(n)$ is nonempty for all $n\geq 3$. Equivalently, it is not known if every $n\geq 3$ is the covering number of some ring. If true, this would be in stark contrast to the corresponding situation for groups: it is shown in \cite{Wer18} that $\sigma(M_2(F_3) ) =7$, but Theorem 3.6. of \cite{Tom97} states that $7$ is not the covering number of any group. As a more refined question, one may ask if, given $n\geq 3$ such that $S(n)$ is nonempty, and given a prime $p<n$, there is always a ring in $S(n)$ of characteristic $p$. 

If $R\in S(n)$ for some $n\geq 3$, then there exists what we have called a good tuple $(R, S_1, \ldots, S_n)$ with $R=\bigcup_i S_i $. It would be interesting to determine how unique this tuple is for given $R$; for $R\in S(3)$ and $R\in S(4)$ there are exactly $\sigma(R)$ maximal subrings. One can also ask if there are any common features of such covers. For example, if $R\neq F_2\times F_2$ and $R\in S(n)$, then must we always have $1_R\in \bigcap\limits_{i=1}^n S_i$?  

It is clear that if $R$ is coverable and all its proper quotients are not, then $R\in S(\sigma(R))$. Based on the known examples, one can speculate whether the converse holds: if $R\in S(\sigma(R))$, does one always have $\sigma(R/I)=\infty$ for every proper ideal $I$? Such a theorem would, if true, imply that a noncommutative ring in any $S(n)$ can only have commutative quotients. 


Following the example of groups, one might consider an expanded notion of coverings in which cosets of subrings are considered. Alternatively, one might restrict the notion of coverings to permit only special classes of subrings. Does doing so affect the existence or basic properties of the corresponding set $S(n)$?


\end{document}